\documentclass[10pt]{article}
\usepackage{amsmath,amsfonts}
\normalbaselines
\newtheorem{teor}{Theorem}

\newtheorem{lema}[teor]{Lemma}
\newtheorem{coro}[teor]{Corollary}
\newtheorem{rem}[teor]{Remark}

\newtheorem{ejem}[teor]{Example}

\newcommand{\caja}{\null\hfill\rule{2mm}{2mm}}

\title{Constant mean curvature spacelike hypersurfaces in Lorentzian warped products and Calabi-Bernstein type problems}

\author{Juan A. Aledo${}^{a}$, Alfonso
Romero${}^{b*}$ and Rafael M. Rubio${}^{c*}$ \\[6mm]
${}^a$ Departamento de Matem\'aticas, E.S.I. Inform\'atica, \\[0.5mm] Universidad de
Castilla-La Mancha, 02071 Albacete, Spain,\\ E-mail\textup{:
\texttt{juanangel.aledo@uclm.es}
}\\[3mm]
${}^b$ Departamento de Geometr\'\i a y Topolog\'\i a, \\ [0.5mm]
Universidad de Granada, 18071 Granada, Spain \\ E-mail\textup{:
\texttt{aromero@ugr.es}} \\[3mm]
${}^c$ Departamento de Matem\'aticas, Campus de Rabanales, \\[0.5mm] Universidad de
C\'ordoba, 14071 C\'ordoba, Spain,\\[0.5mm] E-mail\textup{: \texttt{rmrubio@uco.es}}\\[3mm]}

\date{}

\begin{document}

\maketitle

\thispagestyle{empty}

\begin{abstract}
In this paper we provide several uniqueness and non-existence
results for complete parabolic constant mean curvature spacelike
hypersurfaces in Lorentzian warped products under appropriate
geometric assumptions. As a consequence of this parametric study, we
obtain very general uniqueness and non-existence results for a large
family of uniformly elliptic EDP's, so solving the Calabi-Bernstein
problem in a wide family of spacetimes.
\end{abstract}
\vspace*{5mm}



\section{Introduction}

The study of spacelike hypersurfaces with constant mean curvature
(CMC) in Lorentzian manifolds has attracted the interest of a
considerable group of geometers as evidenced by the amount of works
that it has generated. This is due not only to its mathematic
interest, but also to its relevance in General Relativity;  a
summary of several reasons justifying its interest can be found in
\cite{M-T}. In particular, hypersurfaces of (non-zero) CMC are
particularly suitable for studying the propagation of gravity
radiation \cite{S}. Classical papers dealing with uniqueness results
on CMC spacelike hypersurfaces are, for instance, \cite{BF},
\cite{Ch} and \cite{M-T}. In \cite{BF}, Brill and Flaherty
considered a spatially closed universe, and proved several
uniqueness results on CMC hypersurfaces in the large by assuming
that the Ricci curvature of the spacetime satisfies that
$\overline{{\mathrm{Ric}}}(z,z)>0$ for all timelike vector $z$. In
\cite{M-T}, this energy condition was relaxed by Marden and Tipler
to include, for instance, non-flat vacuum spacetimes. Later, Bartnik
proved in \cite{Bar} very general existence theorems on CMC
spacelike hypersurfaces, and claimed that it would be useful to find
new satisfactory uniqueness results. More recently, in \cite{A-R-S1}
Alias, Romero and Sanchez proved new uniqueness results for CMC
hypersurfaces in the class of spacetimes that they call closed
generalized Robertson-Walker spacetimes (which includes the
spatially closed Robertson-Walker spacetimes), under the Temporal
Convergence Condition (TCC). Finally, in \cite{RRS}, Romero, Rubio
and Salamanca, have provided some uniqueness results for the maximal
case (zero mean curvature) in spatially parabolic generalized
Robertson-Walker spacetimes, which are open models whose fiber is a
parabolic Riemannian manifold.

In the paradigmatic case of CMC hypersurfaces immersed in the
Lorentz-Minkowski space $\mathbb{L}^{n+1}$, $n\geq 2$, there is a
great variety of results from different points of view. One of the
most celebrated results is the solution to the corresponding
Bernstein problem for maximal hypersurfaces, known as
Calabi-Bernstein problem in this Lorentzian context, by Calabi
($n\leq 4$) \cite{Ca}, and Cheng and Yau (arbitrary $n$) \cite{CY}.
As for the case of nonzero constant mean curvature, many nonlinear
examples of complete spacelike hypersurfaces with nonzero constant
mean curvature can be constructed (see for instance \cite{HN},
\cite{IH}, \cite{Tr}). In \cite{AA} the spacelike hyperplanes in
$\mathbb{L}^{n+1}$ are characterized as the only complete CMC
spacelike hypersurfaces which are bounded between two parallel
spacelike hyperplanes. On the other hand, Aiyama \cite{Ai} and Xin
\cite{Xi} simultaneous and independently characterized spacelike
hyperplanes as the only complete CMC spacelike hypersurfaces in
$\mathbb{L}^{n+1}$ whose image under the Gauss map is bounded in the
hyperbolic $n$-space (see also \cite{Pa} for a weaker first version
of this result given by Palmer). Recall that the Gauss application
$N$ of a spacelike hypersurface $M$ immersed in $\mathbb{L}^{n+1}$
can be thought as an application from $M$ into the hyperbolic space
$\mathbb{H}^{n}\subset \mathbb{L}^{n+1}$. Thus, the Gauss
application is bounded if and only if the hyperbolic angle between
$N$ and the oriented timelike axis is also bounded.

In the general case of a spacelike hypersurface in a Lorentzian
manifold $\overline{M}$, the Gauss map of $M$ can be globally
defined provided $\overline{M}$ is time-orientable. Although in this
context it has not sense talking about bounded Gauss application,
once we choose a unitary timelike vector field globally defined on
$\overline{M}$ (compatible with the time-orientation), the notion of
hyperbolic angle can be naturally defined (see Section \ref{s2} for
the details). Thus, the assumption of bounded hyperbolic angle is
the natural extension of the one used by Aiyama and Xin in their
result, and actually it has been used in this context \cite{CRR1}, \cite{CRR2}.

In this paper we consider a wide family of Lorentzian manifolds,
given by the warped product of an 1-dimensional manifold endowed
with a negative definite metric and an $n$-dimensional ($n\geq 2$)
Riemannian manifold which, in general, will be taken complete and
non-compact. Note that the classical spatially open Robertson-Walker
cosmological models are included in that family.

In these ambient spaces there exists a distinguished unitary
timelike vector field globally defined which allows to naturally
define the notion of hyperbolic angle for every immersed spacelike
hypersurface. Thus, one of our main aims will be to provide
characterizations of complete spacelike hypersurfaces with bounded
hyperbolic angle under suitable geometric hypothesis (for instance,
energy-type conditions) on both the ambient space and the
hypersurface. As will be pointed out, the family of warped products
for which our results are applicable is very large and contains
notable examples.

More precisely,  given an $n(\geq 2)$-dimensional (connected)
Riemannian manifold $(F,g_{_F})$, an open interval $I$ in
$\mathbb{R}$ and a positive smooth function $f$ defined on $I$, the
product manifold $I \times F$ endowed with the Lorentzian metric
\[
\bar{g} = -\pi^*_{_I} (dt^2) +f(\pi_{_I})^2 \, \pi_{_F}^* (g_{_F})
\, ,
\]
where $\pi_{_I}$ and $\pi_{_F}$ denote the projections onto $I$ and
$F$, respectively, is a \emph{Lorentzian warped product} in the
sense of \cite[cap. 7]{O'N}. This kind of Lorentzian manifolds are
also known as Generalized Robertson-Walker (GRW) spacetimes in the
physical context \cite{A-R-S1}. Along this paper we will represent
this $(n+1)$-dimensional Lorentzian manifold by $\overline{M}= I
\times_f F$. When $n=3$ and the fiber $F$ has constant sectional
curvature, $\overline{M}= I \times_f F$ is known as a
Robertson-Walker (RW) spacetime. Note that a RW spacetime obeys the
\emph{cosmological principle}, i.e. it is spatially homogeneous and
spatially isotropic, at least locally.  Thus, GRW spacetimes widely
extend to RW spacetimes and include, for instance, the
Lorentz-Minkowski spacetime, the Einstein-de Sitter spacetime, the
Friedmann cosmological models, the static Einstein spacetime and the
de Sitter spacetime. GRW spacetimes are useful to analyze if a
property of a RW spacetime $M$ is \emph{stable}, i.e. if it remains
true for spacetimes close to $M$ in a certain topology defined on a
suitable family of spacetimes \cite{Geroch}. In fact, a deformation
$s \mapsto g_{_F}^{(s)}$ of the metric of $F$ provides a one
parameter family of GRW spacetimes close to $M$ when $s$ approaches
to $0$. Note that a conformal change of the metric of a GRW
spacetime, with a conformal factor which only depends on $t$,
produces a new GRW spacetime. On the other hand, a GRW spacetime is
not necessarily spatially homogeneous. Recall that spatial
homogeneity seems appropriate just as a rough approach to consider
the universe in the large. However, this assumption could not be
realistic when the universe is considered in a more accurate scale.
Thus, a GRW spacetime could be a suitable spacetime to model a
universe with inhomogeneous spacelike geometry \cite{Ra-Sch}.

When the fiber $F$ is a compact (without boundary) Riemannian
manifold, the Lorentzian warped product $\overline{M}= I \times_f F$
is said to be spatially \emph{closed}. On the other hand, if $F$ is
complete and non compact, we will say that $\overline{M}$ is
spatially \emph{open}. In this last case, if moreover $F$ is
parabolic then $\overline{M}$ is said to be \emph{spatially
parabolic} \cite{RRS}. The open case is especially interesting
since, unlike that in the closed one, it can be compatible with the
inflation hypothesis and the holographic principle \cite{Bak-Rey},
\cite{Bo}. Moreover, in a spatially open Lorentzian warped product
$\overline{M}= I \times_f F$, the boundedness of the hyperbolic
angle of a spacelike hypersurface $M$ has a physical interpretation.
In fact, consider the unitary normal vector field $N$ on $M$ and the
unit timelike vector field $\mathcal{T}_{p}:=-\partial_t$ (the sign
minus depends on the chosen time orientation). Along $M$ there exist
two families of \emph{instantaneous observers} $\mathcal{T}_{p}$,
$p\in M$, and the normal observers $N_p$. The quantities $$\cosh
\varphi (p) \quad \text{and} \quad
v(p):=\left(\frac{1}{\cosh\varphi(p)}\right)\, N_p^F,$$ where
$N_p^F$ is the projection of $N_p$ onto $F$ and $\varphi$ the
hyperbolic angle of $M$, are respectively the \emph{energy} and the
\emph{velocity} that $\mathcal{T}_p$ measures for $N_p$, and we have
$\vert v\vert= \tanh\varphi$ on $M$, \cite[pp. 45,67]{Sa-Wu}.
Therefore the \emph{relative speed function} $\vert v\vert$ is
bounded on $M$ and, hence, it does not approach to the speed of
light in vacuum.

Our paper is organized as follows. In Section \ref{s2} we introduce
the notation to be used for spacelike hypersurfaces in Lorentzian
warped products. Section \ref{parabolicity} is devoted to revise
some results regarding the parabolicity of Riemannian manifolds,
paying special attention to show under which assumption this
condition can be deduced for a spacelike hypersurface of a
Lorentzian warped product from suitable assumptions of the ambient
space. In Section \ref{tr} we provide an inequality (see Lemma \ref
{tocho}) involving the hyperbolic angle of a CMC spacelike
hypersurface immersed in a Lorentzian warping product satisfying the
TCC. This inequality will be the key for obtaining our results. In
Section \ref{pr} we present several uniqueness and non-existence
results for complete parabolic CMC spacelike hypersurfaces in
Lorentzian warped products under appropriate assumptions. In
particular, we widely generalized the results of Aiyama-Xin
commented above. Finally, in Section \ref{CBr} we apply our
parametric results to the study of several Calabi-Bernstein type
problems in this context. Observe that, unlike that in the case of
entire graph into a Riemannian product space, an entire spacelike
graph in a Lorentzian (or warped Lorentzian) product is no
necessarily complete, in the sense that the induced Riemannian
metric is not necessarily complete on the graph. As a non direct
application of the parametric case, we obtain very general
uniqueness and non-existence results  for a wide family of uniformly
ellyptic EDP's (see Equation (E.1)+(E.2)).

\section{Preliminaries}
\label{s2} \noindent

Let $(F,g_{_F})$ be an $n(\geq 2)$-dimensional (connected)
Riemannian manifold, $I$ an open interval in $\mathbb{R}$ and $f$ a
positive smooth function defined on $I$. Then, the product manifold
$I \times F$ endowed with the Lorentzian metric
\begin{equation}\label{metrica}
\bar{g} = -\pi^*_{_I} (dt^2) +f(\pi_{_I})^2 \, \pi_{_F}^* (g_{_F})
\, ,
\end{equation}
where $\pi_{_I}$ and $\pi_{_F}$ denote the projections onto $I$ and
$F$, respectively, is called a \emph{Lorentzian warped product} with
\emph{fiber} $(F,g_{_F})$, \emph{base} $(I,-dt^2)$ and \emph{warping
function} $f$. Along this paper we will represent this
$(n+1)$-dimensional Lorentzian manifold by $\overline{M}= I \times_f
F$.

In any Lorentzian warped product $\overline{M}=I\times_f F$, the
coordinate vector field $\partial_t:=\partial/\partial t$ is
(unitary) timelike, and hence $\overline{M}$ is time-orientable.
Thus,  the timelike vector field $K: = f({\pi}_I)\,\partial_t$ is
also timelike. Moreover, from the relationship between the
Levi-Civita connections of $\overline{M}$ and those of the base and
the fiber \cite[Cor. 7.35]{O'N}, it follows that
\begin{equation}\label{conexion} \overline{\nabla}_XK =
f'({\pi}_I)\,X
\end{equation}
for any $X\in \mathfrak{X}(\overline{M})$, where $\overline{\nabla}$
is the Levi-Civita connection of the Lorentzian metric
(\ref{metrica}).

Given an $n$-dimensional manifold $M$, an immersion $\psi: M
\rightarrow \overline{M}$ is said to be \emph{spacelike} if the
Lorentzian metric (\ref{metrica}) induces, via $\psi$, a Riemannian
metric $g_{_M}$ on $M$. In this case, $M$ is called a spacelike
hypersurface. We will denote by $\tau=\pi_I\circ \psi$ the
restriction of $\pi_I$ along $\psi$.

The time-orientation of $\overline{M}$ allows to take, for each
spacelike hypersurface $M$ in $\overline{M}$, a unique unitary
timelike vector field $N \in \mathfrak{X}^\bot(M)$ globally defined
on $M$ with the same time-orientation as $-\partial_t$, i.e., such
that $\bar{g}(N,K)\geq f(\tau):=f\circ \tau>0$ and $\bar{g}(N,K)=
f(\tau)$ at a point $p\in M$ if and only if $N = -\partial_t$ at
$p$. We will denote by $A$ the shape operator associated to $N$.
Then the \emph{mean curvature function} associated to $N$ is given
by $H:= -(1/n) \mathrm{trace}(A)$. As is well-known, the mean
curvature is constant if and only if the spacelike hypersurface is,
locally, a critical point of the $n$-dimensional area functional for
compactly supported normal variations, under certain constraint of
the volume. When the mean curvature vanishes identically, the
spacelike hypersurface is called a \emph{maximal} hypersurface.

For a spacelike hypersurface $\psi: M \rightarrow \overline{M}$ with
Gauss map $N$, the \emph{hyperbolic angle} $\varphi$, at any point
of $M$, between the unit timelike vectors $N$ and $-\partial_t$, is
given by $\bar{g}(N,\partial_t)=\cosh \varphi$. By simplicity,
throughout this paper we will refer to $\varphi$  as the
\emph{hyperbolic angle function} on $M$.

In any Lorentzian warped product $\overline{M}= I \times_f F$ there
is a remarkable family of spacelike hypersurfaces, namely its
spacelike slices $\{t_{0}\}\times F$, $t_{0}\in I$. It can be easily
seen that a spacelike hypersurface in $\overline{M}$ is a (piece of)
spacelike slice if and only if the function $\tau$ is constant.
Furthermore, a spacelike hypersurface in $\overline{M}$ is a (piece
of) spacelike slice if and only if the hyperbolic angle $\varphi$
vanishes idenentically. The shape operator of the spacelike slice
$\tau=t_{0}$ is given by $A=f'(t_{0})/f(t_{0})\,I$, where $I$
denotes the identity transformation, and therefore its (constant)
mean curvature is $H=- f'(t_{0})/f(t_{0})$. Thus, a spacelike slice
is maximal if and only if $f'(t_{0})=0$ (and hence, totally
geodesic).

We will say that the spacelike hypersurface is contained in a
\emph{slab}, if it is contained between two slices. Analogously, we
will say that the hypersurface is contained in an \emph{open slab}
if it is contained in the unbounded region determined for a slice,
i.e. if there exists $t_{0}\in I$ such that $\tau>t_0$ (resp.
$\tau<t_0$) on $M$.

\section{Parabolic Riemannian manifolds}\label{parabolicity}
Because of its appearance  in so many applications (Laplace
equation, Helmholtz equation, etc., see for instance \cite{So}) one
of the central problems in mathematics is to understand the equation
$Lu=f$ on a Riemannian manifold, where $Lu=\Delta u+au$ denotes the
Schrodinger operator and $a$, $f$ are two given function. More
generally, the family of equations $\Delta u+ah(u)=f$ naturally
arises in geometry in several contexts. For instance, given $(M,g)$
a 2-Riemannian manifold and $g_1=e^{2u}g$ a metric pointwise
conformal to $g$, if $K$ and $K_1$ denote the Gaussian curvatures
for $g$ and $g_1$, respectively, then $\Delta u+K_1e^{2u}=K$ where
$\Delta$ is the Laplacian operator for the metric $g$.

Although this kind of equations is fairly well studied on a compact
manifold without boundary, very little is known in the complete but
non compact case. A non-compact Riemannian manifold is said to be
\emph{parabolic} if it does not admit non-constant positive
superharmonic functions (see \cite{Kazdan}, for instance). Hence,
the rich geometric analysis of compact Riemannian manifolds is
preserved, in some sense, for complete parabolic manifolds.
Moreover, the parabolicity is part of the more general problem of
understanding the lack of uniqueness of solutions to the
aforementioned equations.

In the two dimensional case, this notion is very close to the
classical parabolicity for Riemann surfaces. Moreover, it is
strongly related to the behavior of the Gaussian curvature $K$ of
the surface, since every complete Riemannian  surface with $K\geq 0$
is parabolic (see \cite{A}).

Recall also that if $S$ is a complete surface and there exists a
point $p_0\in S$ and a positive constant $r_0$ such that $K(p)\geq
\frac{-1}{r^2(p)\log r(p)}$ for all $p\in S$ such that $r(p)={\rm
dist}(p_0,p)\geq r_0$, then $S$ is parabolic \cite{GW}. On the other
hand,  if the negative part of $K$  is integrable on a complete
surface $S$, then $S$ is parabolic \cite{Li}.

For higher dimensions, parabolicity of Riemannian manifolds is quite
different and, in particular, it has not a so direct relation with
the sectional curvature of the manifold. In fact, the Euclidean
space $\mathbb{R}^n$ is parabolic if and only if $n\leq 2$. On the
other hand, if $(M_1,g_1)$ is a compact Riemannian manifold and
$(M_2,g_2)$ is a parabolic Riemannian manifold, then $M_1\times M_2$
endowed with the product metric $g_1 + g_2$ is parabolic,
\cite{Kazdan}. This also works for certain warped products (see
\cite{Kazdan} for the details). In particular, the product of a
compact Riemannian manifold and the real line $\mathbb{R}$  or the
Euclidean plane $\mathbb{R}^2$ is always a parabolic Riemannian
manifold.

Parabolicity is also closely related with the volume growth of the
geodesic balls in an $n(\geq 2)$-dimensional non-compact complete
Riemannian manifold. Indeed, if the volume growth of the geodesic
balls is \emph{moderate}, then the Riemannian manifold is parabolic
(see \cite{Kar} for the details).

The study of immersed Riemannian submanifolds (with the induced
metric) in Lorentzian ambient spaces is important not only from a
mathematical point of view, but also due to its relevant role in
General Relativity. In particular, the case with codimension 1
(spacelike hypersurfaces) is crucial. Lorentzian warped products
which admit a complete parabolic spacelike hypersurface have been
studied in \cite{RRS}, where the following result is proved:

\begin{quote} {\it Let $M$ be a complete spacelike hypersurface
in a Lorentzian warped product $\overline{M}= I \times_f F$, whose
fiber has parabolic universal Riemannian covering. If the hyperbolic
angle of $M$ is bounded and the restriction $f(\tau)$ on $M$ of the
warping function $f$ satisfies:
\begin{itemize}
\item[i)] $\sup f(\tau)<\infty$, and
\item[ii)] $\inf f(\tau)>0,$
\end{itemize}
then, $M$ is parabolic.

Conversely, let $\overline{M}= I \times_f F$ be a Lorentzian warped
product such that its warping function $f$ satisfies i) and ii). If
$\overline{M}$ admits a simply connected parabolic spacelike
hypersurface $M$ whose hyperbolic angle is bounded, then the
universal Riemannian covering of the fiber of $\overline{M}$ is
parabolic. }
\end{quote}

Observe that in the particular case of Lorentzian warped products $\overline{M}= I \times_f F$
with a simply connected complete parabolic Riemannian fiber, the
assumption on the universal Riemannian covering is trivially
satisfied. Therefore, every Lorentzian warped product with a simply
connected complete parabolic Riemannian fiber, whose warping
function is bounded from above and has non-zero infimum, satisfies
the assumptions on the ambient space in the previous result. This is
the case, for instance, of Lorentzian products (i.e. with warping
function equal to 1) whose fiber is a simply connected complete
parabolic Riemannian manifold.

\begin{ejem}
Taking the previous considerations into account, a complete
spacelike hypersurface with bounded hyperbolic angle in the
Lorentzian warped product $I\times_f
(\mathbb{S}^2\times\mathbb{R}^2)$, where the warping function
satisfies i) and ii), must be parabolic. Note also that on the
2-sphere $\mathbb{S}^2$ any Riemannian metric can be considered.
\end{ejem}

\section{A technical result} \label{tr}
Let $\psi: M \rightarrow \overline{M}$ be an $n$-dimensional
spacelike hypersurface immersed in a Lorentzian warped product
$\overline{M}= I \times_f F$. If we denote by
\[
\partial_t^T:= \partial_t+\overline{g}(N,\partial_t)N
\]
the tangential component of $\partial_t$ along $\psi$, then it is
easy to check that the gradient of $\tau$ on $M$ is given by
\begin{equation}\label{part}
\nabla \tau=-\partial_t^T
\end{equation}
and so
\begin{equation}\label{sinh}
|\nabla \tau|^2=g_{_M}(\nabla \tau,\nabla \tau)=\sinh^2 \varphi.
\end{equation}
Moreover, if we put $K^T=K+\overline{g}(K,N)N$ the tangential
component of $K$ along $\psi$, a direct computation from
(\ref{conexion}) gives
\begin{equation}\label{gradcosh}
\nabla \overline{g}(K,N)=-AK^T
\end{equation}
where we have used (\ref{part}), and also
\[
\nabla \cosh \varphi=-A\partial_t^T+\frac{f'(\tau)}{f(\tau)}\overline{g}(N,\partial_t)\partial_t^T.
\]

On the other hand, if we represent by $\nabla$ the Levi-Civita
connection of the metric $g_{_M}$, then the Gauss and Weingarten
formulas for the immersion $\psi$ are given, respectively, by
\begin{equation}\label{GF}
\overline{\nabla}_X Y=\nabla_X Y-g_{_M}(AX,Y)N
\end{equation}
and
\begin{equation}\label{WF}
AX=-\overline{\nabla}_X N,
\end{equation}
where $X,Y\in\mathfrak{X}({M})$. Then,  we get from
(\ref{conexion}), (\ref{GF}) and (\ref{WF}), that
\begin{equation}\label{KT}
\nabla_X K^T=-f(\tau)\overline{g}(N,\partial_t)AX+f'(\tau)X
\end{equation}
where $X\in\mathfrak{X}({M})$ and $f'(\tau):=f'\circ \tau$. Since
also $K^T=f(\tau)\partial_t^T$, it follows from (\ref{part}) and
(\ref{KT}) that the Laplacian of $\tau$ on $M$ is
\begin{equation}\label{laptau}
\Delta \tau=-\frac{f'(\tau)}{f(\tau)}\{n+|\nabla
\tau|^2\}-nH\overline{g}(N,\partial_t).
\end{equation}

Consequently
\begin{eqnarray}
\Delta f(\tau) & = & f'(\tau) \Delta \tau+f''(\tau)|\nabla \tau|^2
\nonumber \\
& = & -\frac{f'(\tau)^2}{f(\tau)} \, n + |\nabla \tau|^2 f(\tau)
(\log f)''(\tau)-n H f'(\tau)\cosh \varphi \label{lapftau}
\end{eqnarray}
and so
\begin{eqnarray}
\Delta \left( f(\tau) \cosh \varphi\right)
& = &
\cosh \varphi \, \, \Delta f(\tau)+f(\tau) \, \, \Delta \cosh \varphi+2g_{_M}(\nabla f(\tau), \nabla \cosh \varphi) \nonumber \\
& = & -\frac{f'(\tau)^2}{f(\tau)} \, n\cosh \varphi +  f(\tau) \cosh \varphi \, \, \sinh^2 \varphi \, (\log f)''(\tau)-n H f'(\tau)\cosh^2 \varphi \nonumber \\
& & +f(\tau) \, \, \Delta \cosh \varphi+2g_{_M}(A\partial_t^T,
\partial_t^T)-2\frac{f'(\tau)^2}{f(\tau)}\, \, \cosh \varphi \, \,
\sinh^2 \varphi, \label{lapfcosh}
\end{eqnarray}
where we have used (\ref{part}), (\ref{sinh}) and (\ref{gradcosh}).

On the other hand, if we assume that $M$ is a CMC hypersurface, we
get from the Codazzi equation for $M$ that (see \cite[Eq.
8]{A-R-S1})
\begin{equation}\label{ARS1995}
\Delta \left( f(\tau) \cosh \varphi\right)=\Delta
\overline{g}(K,N)=\overline{{\rm
Ric}}(K^T,N)+f'(\tau)nH+f(\tau)\cosh\varphi \,\, {\rm trace}(A^2)
\end{equation}
where $\overline{{\rm Ric}}$ stands for the Ricci tensor on
$\overline{M}$. Therefore, from (\ref{lapfcosh}) and (\ref{ARS1995})
we have
\begin{eqnarray}
\overline{{\rm Ric}}(K^T,N) & = & - f'(\tau)nH- f(\tau)\cosh\varphi \,\, {\rm trace}(A^2) \nonumber \\
& & -\frac{f'(\tau)^2}{f(\tau)} \, n\cosh \varphi +  f(\tau) \cosh \varphi \, \, \sinh^2 \varphi \, (\log f)''(\tau)-n H f'(\tau)\cosh^2 \varphi \nonumber \\
& & +f(\tau) \, \, \Delta \cosh \varphi+2g_{_M}(A\partial_t^T,
\partial_t^T)-2\frac{f'(\tau)^2}{f(\tau)}\, \, \cosh \varphi \, \,
\sinh^2 \varphi. \label{Ric1}
\end{eqnarray}

If we put $N=N_{_F}-\overline{g}(N,\partial_t^T)\partial_t$, where
$N_{_F}$ denotes the projection of $N$ on the fiber $F$, it is easy
to obtain from (\ref{metrica}) that
\begin{equation}\label{senh2}
\sinh^2\varphi=f(\tau)^2 \,\, g_{_F}(N_{_F},N_{_F}).
\end{equation}
Besides, from \cite[Chapter 7, Corollary 43]{O'N} we know that
\begin{equation}\label{ricpartial}
\overline{{\rm
Ric}}(\partial_t,\partial_t)=-n\frac{f''(\tau)}{f(\tau)}
\end{equation}
and
\begin{equation}\label{ricNF}
\overline{{\rm Ric}}(N_{_F},N_{_F})={\rm
Ric}^F(N_{_F},N_{_F})+\sinh^2\varphi
\left(\frac{f''(\tau)}{f(\tau)}+(n-1)\frac{f'(\tau)^2}{f(\tau)^2}
\right)
\end{equation}
where ${\rm Ric}^F$ stands for the Ricci tensor on $F$ and we have
used (\ref{senh2}). Then, from (\ref{ricpartial}) and (\ref{ricNF})
we obtain
\begin{eqnarray}
\overline{{\rm Ric}}(K^T,N) & = & f(\tau)\cosh\varphi \,\, \overline{{\rm Ric}}(N_{_F},N_{_F})-f\cosh \varphi \,\, \sinh^2 \varphi \,\, \overline{{\rm Ric}}(\partial_t,\partial_t) \nonumber \\
& = & f(\tau)\cosh\varphi\,\, {\rm Ric}^F(N_{_F},N_{_F})-(n-1)
f(\tau)\cosh\varphi\,\,\sinh^2\varphi \,\,(\log f)''(\tau).
\label{Ric2}
\end{eqnarray}
Finally, from (\ref{Ric1}) and (\ref{Ric2}) we get
\begin{eqnarray}
\Delta \cosh \varphi & = & nH\frac{f'(\tau)}{f(\tau)}\left(1+\cosh^2\varphi\right)
+\cosh\varphi\left( {\rm Ric}^F(N_{_F},N_{_F})-n \sinh^2\varphi \,\,(\log f)''(\tau)\right) \nonumber \\
&&+ \frac{f'(\tau)^2}{f(\tau)}\, \, \cosh \varphi \left(n+2\sinh^2
\varphi\right) +\cosh\varphi \,\, {\rm
trace}(A^2)-2\frac{f'(\tau)}{f(\tau)}\,\,g_{_M}(A\partial_t^T,
\partial_t^T). \label{wegotit}
\end{eqnarray}

On the other hand, the square algebraic trace-norm of the Hessian
tensor of $\tau$ is just
\[
\vert {\rm Hess}(\tau)\vert^2={\rm trace} (H_\tau\circ H_\tau),
\]
where $H_\tau$ denotes the operator defined by
$g_{_M}(H_\tau(X),Y):={\rm Hess}(X,Y)$ for all vector fields
$X,Y\in\mathfrak{X}({M})$ on $M$, and Hess is the Hessian operator
for the metric $g_{_M}$.

By taking the tangential component in (\ref{conexion}) and using
(\ref{part}), we get that
\begin{eqnarray}
\vert{\rm Hess}(\tau)\vert^ 2 & = & \frac{f'(\tau)^ 2}{f(\tau)^
2}\left((n-1)+\cosh^ 4\varphi\right)+\cosh^ 2\varphi\,{\rm
trace}(A^ 2)+2nH\frac{f'(\tau)}{f(\tau)}\cosh\varphi \nonumber \\
&& -2\frac{f'(\tau)}{f(\tau)}\cosh\varphi \, g_{_M}(A\partial_t^
T,\partial_t^ T). \label{hess}
\end{eqnarray}
Since $\vert{\rm Hess}(\tau)\vert^ 2\geq 0$, it is a straightforward
computation to obtain, making use of (\ref{wegotit}) and
(\ref{hess}), that
\begin{eqnarray}
\cosh \varphi\,\, \Delta \cosh \varphi & \geq &
nH\frac{f'(\tau)}{f(\tau)}\cosh\varphi\sinh^2\varphi+\cosh^2\varphi\left(
{\rm Ric}^F(N_{_F},N_{_F})-n \sinh^2\varphi \,\,(\log
f)''(\tau)\right)
\nonumber \\
&&
\hspace{-0.4cm}+n\frac{f'(\tau)^2}{f(\tau)^2}\cosh^2\varphi+2\frac{f'(\tau)^2}{f(\tau)^2}\cosh^2\varphi\sinh^2\varphi-
\frac{f'(\tau)^2}{f(\tau)^2}\left((n-1)+\cosh^4\varphi\right).
\label{wegotit2}
\end{eqnarray}

Now, let us assume that the ambient spacetime satisfies the Timelike
Convergence Condition (TCC). Recall  that a Lorentzian manifold
obeys the Timelike Convergence Condition if its Ricci tensor
$\overline{\rm Ric}$ satisfies $\overline{\rm Ric}(z,z) \geq 0$, for
all timelike vector $z$ .Observe that (see \cite[Section 2]{S99}),
 in our context, TCC is equivalent to
\[
f''\leq 0, \qquad {\rm Ric}^F(X,X)\geq n (ff''-f'^2)g_{_F}(X,X)
\]
for all $X$ tangent to the fiber $F$. In particular, since
\[
(\log f)''(\tau)=\frac{f(\tau)f''(\tau)-f'(\tau)^2}{f(\tau)^2}
\]
and using also (\ref{senh2}), it follows that
\begin{equation}\label{RF}
{\rm Ric}^F(N_{_F},N_{_F})-(n-1)(\log f)''(\tau)\sinh^2\varphi\geq
0,
\end{equation}
which jointly with (\ref{wegotit2}) yields
\begin{eqnarray*}
\cosh \varphi\,\, \Delta \cosh \varphi & \geq &
nH\frac{f'(\tau)}{f(\tau)}\cosh\varphi\sinh^2\varphi
+n\frac{f'(\tau)^2}{f(\tau)^2}\cosh^2\varphi
+3\frac{f'(\tau)^2}{f(\tau)^2}\cosh^2\varphi\sinh^2\varphi  \\
&&
-(n-1)\frac{f'(\tau)^2}{f(\tau)^2}-\frac{f'(\tau)^2}{f(\tau)^2}\cosh^2\varphi\left(1+\sinh^2\varphi\right)
\\
& \geq &
\sinh^2\varphi\Big(\frac{n}{2}H+\frac{f'(\tau)}{f(\tau)}\cosh\varphi\Big)^2
+n\sinh^2\varphi\Big(\frac{f'(\tau)^2}{f(\tau)^2}-\frac{n}{4}H^2\Big)
+\frac{f'(\tau)^2}{f(\tau)^2}\sinh^4\varphi.
\end{eqnarray*}

Summing up, we have proved the following result

\begin{lema} \label{tocho} Let $\psi: M \rightarrow \overline{M}$ be an $n$-dimensional CMC spacelike
hypersurface immersed in a Lorentzian warped product $\overline{M}=
I \times_f F$ satisfying the TCC. Then
\begin{eqnarray}\label{laplacosh}
\cosh \varphi\,\, \Delta \cosh \varphi & \geq &
\sinh^2\varphi\Big(\frac{n}{2}H+\frac{f'(\tau)}{f(\tau)}\cosh\varphi\Big)^2 \nonumber \\
&&
+n\sinh^2\varphi\Big(\frac{f'(\tau)^2}{f(\tau)^2}-\frac{n}{4}H^2\Big)
+\frac{f'(\tau)^2}{f(\tau)^2}\sinh^4\varphi.\label{fundamental}
\end{eqnarray}
\end{lema}

\section{Parametric type results}\label{pr}

Our aim is to use Lemma \ref{tocho} in order to, under appropriate
assumptions, deduce that the hyperbolic angle $\varphi$ vanishes
identically (and so $M$ is a spacelike slice) or give non-existence
results. In this sense,  we will need to ask the hypersurface to be
parabolic. As explained in Section \ref{parabolicity}, this property
follows automatically under certain natural hypothesis on the
ambient space $\overline{M}= I \times_f F$. However, for the sake of
clarity, we have decided to state our results under the assumption
of parabolicity of the hypersurface.

On the other hand, we will use of the following technical result
(see, for instance, \cite[Lemma 3.1]{Ro-Ru1}) which will allow us to
obtain an upper bound for the integral of the squared length of
$\nabla \cosh \varphi$ on a geodesic ball of $M$.
\begin{lema}\label{lemma1}
Let $M$ be an $n(\geq 2)$-dimensional Riemannian manifold and
consider $v \in C^2(M)$ satisfying $v\Delta v\geq 0$. Let $B_R$ be a
geodesic ball of radius $R$ around $p\in M$. For any $r$ such that
$0<r<R$ we have
\[
\int_{B_r}\vert\nabla
v\vert^2\,dV\leq\frac{4\sup_{B_{R}}v^2}{\mu_{r,R}},
\]
where $B_r$
denotes the geodesic ball of radius $r$ around $p\in M$ and
$\frac{1}{\mu_{r,R}}$ is the capacity of the annulus
$B_R\setminus\bar B_r$.
\end{lema}

Let $\psi: M \rightarrow \overline{M}$ be an $n$-dimensional CMC
spacelike immersed in a GRW-spacetime $\overline{M}= I \times_f F$
satisfying TCC, and let us assume that the constant mean curvature
of $M$  verifies that
\begin{equation}\label{inequality}
H^2\leq \frac{f'(\tau)^2}{f(\tau)^2}.
\end{equation}
Under these assumptions, from (\ref{fundamental}) and provided that
$n\leq 4$ we get that
\[
\cosh \varphi\,\, \Delta \cosh \varphi\geq 0.
\]
Then, as a consequence of Lemma \ref{lemma1} we have the following
local estimation for the integral of the squared length of $\nabla
\cosh \varphi$ on an arbitrary geodesic ball of $M$.

\begin{teor}
Let $\overline{M}= I \times_f F$ be a Lorentzian warped product with
dimension $n+1\leq 5$ satisfying the TCC. Let $\psi: M \rightarrow
\overline{M}$ be a CMC spacelike hypersurface whose mean curvature
is such that
\[
H^2\leq \frac{f'(\tau)^2}{f(\tau)^2}.
\]
If $B_R$ is a geodesic ball of radius $r$ around $p\in M$, for any
$r$ such that $0<r<R$, then the function $\cosh{\varphi}$ satisfies
\[
\int_{D_r}\vert\nabla \cosh{\varphi}\vert^2\,dV\leq
\frac{C}{\mu_{r,R}},
\]
where $B_r$ is the geodesic ball of radius $r$ around $p\in M$,
$\frac{1}{\mu_{r,R}}$ is the capacity of the annulus
$A_{r,R}:=B_R\setminus\bar{B}_r$ and $C=C(p,R)>0$ is a constant.
\end{teor}

Let us also assume that the hyperbolic angle $\varphi$ of $M$ is
bounded. Then, there exists a positive constant $C$ such that
$\cosh^2\varphi\leq C$ on $M$. Thus, if we apply Lemma \ref{lemma1}
to the function $v=\cosh\varphi$, we have for  a geodesic ball $B_R$
of radius $R$ around $p\in M$, that for any $r$ such that $0<r<R$
the function $\cosh{\varphi}$ satisfies
\[
\int_{B_r}\vert\nabla \cosh{\varphi}\vert^2\,dV\leq
\frac{4C}{\mu_{r,R}}.
\]

Now, if we also assume that $M$ is parabolic we have that
\[
\lim_{R\rightarrow \infty} \frac{1}{\mu_{r,R}}=0,
\]
that is, $\vert\nabla \cosh{\varphi}\vert^2$ vanishes identically on
$M$ and so $\cosh{\varphi}$ is constant on $M$. Observe that, if
$H\neq 0$, from (\ref{inequality}) we have that $f'(\tau)\neq 0$.
Finally, from (\ref{fundamental}) we conclude that $\varphi=0$ and
so $M$ is a (piece of) spacelike slice.

Summing up, we can state the following result:

\begin{teor} \label{t4}
Let $\overline{M}= I \times_f F$ be a Lorentzian warped product with
dimension $n+1\leq 5$ satisfying the TCC. Let $\psi: M \rightarrow
\overline{M}$ be a complete parabolic CMC spacelike hypersurface
with non zero mean curvature $H$ such that
\[
H^2\leq \frac{f'(\tau)^2}{f(\tau)^2}
\]
and whose hyperbolic angle is bounded. Then $M$ must be a spacelike
slice.
\end{teor}

Recall that the (constant) mean curvature of the spacelike slice
$\tau=t_{0}$ is given by $H=- f'(t_{0})/f(t_{0})$. Bearing this in
mind, we are able to give the following non existence result as a
direct consequence of Lemma \ref{tocho}
\begin{teor}\label{none}
Let $\overline{M}= I \times_f F$ be a Lorentzian warped product with
dimension $n+1>5$ satisfying the TCC. Then there is no complete
parabolic CMC spacelike hyperfaces with non zero mean curvature $H$
such that
\[
H^2\leq \frac{4}{n}\frac{f'(\tau)^2}{f(\tau)^2}
\]
and whose hyperbolic angle is bounded.
\end{teor}

\begin{teor} \label{maximales}
Let $\overline{M}= I \times_f F$ be a Lorentzian warped product with
dimension $n+1>2$ satisfying the TCC. Then, every parabolic complete
maximal hypersurface whose hyperbolic angle is bounded and such that
$\sup f(\tau)<\infty$, must be totally geodesic. Moreover, if $M$ is
contained in an open slab, then $M$ must be a spacelike slice
$\tau=t_0$, with $f'(t_0)=0$.
\end{teor}
\begin{proof}
Let $\psi: M \rightarrow \overline{M}$ be a complete parabolic
maximal hypersurface whose hyperbolic angle is bounded and such that
$\sup f(\tau)<\infty$. Since $\overline{M}$ satisfies the TCC, we
get from (\ref{Ric2}) that $\overline{{\rm Ric}}(K^T,N)\geq 0$.
Hence, from (\ref{ARS1995}) we deduce that the positive function
$f(\tau)\cosh\varphi$ is subharmonic. Thus, since $M$ is parabolic
and $f(\tau)\cosh\varphi$ is bounded, it must be constant. Then,
using again (\ref{ARS1995}), we obtain that ${\rm trace}(A^2)=0$ and
so $M$ is totally geodesic.

On the other hand, from Lemma \ref{tocho} we know that
$\cosh\varphi$ is constant, and so $f(\tau)$ must be also constant.
Then, from (\ref{lapftau}) it follows that $f'(\tau)$ vanishes
identically on $M$. Now, using (\ref{laptau}) we deduce that $\tau$
is harmonic and, since it is bounded from above or from below
because $M$ is contained in an open slab, we conclude that $\tau$ is
constant and $M$ is a spacelike slice. \hfill{$\Box$}
\end{proof}

For a Lorentzian product space $\overline{M}=I\times F$ (i.e.
$f=1$), TCC reduces to
\[
{\rm Ric}^F(X,X)\geq 0
\]
for all $X$ tangent to the fiber $F$; that is, $\overline{M}$
satisfies TCC if and only if the fiber $F$ has non-negative Ricci
curvature.

As a consequence of Theorem \ref{maximales} we have

\begin{coro} \label{maximales2}
Let $\overline{M}= I \times F$ be a Lorentzian product with
dimension $n+1>2$ whose fiber has non-negative Ricci curvature.
Then, every complete parabolic maximal hypersurface whose hyperbolic
angle is bounded must be totally geodesic. Moreover, if $M$ is
contained in an open slab, then $M$ must be a spacelike slice
$\tau=t_0$.
\end{coro}

\begin{rem} Observe that the boundedness assumption on $M$ cannot be remove
in Theorem \ref{maximales} and Corollary \ref{maximales2} as shows
the known classical Calabi-Bernstein's theorem in the
Lorentz-Minkowski space.
\end{rem}

Note that for a Lorentzian warped product satisfying the TCC, if
there exists $t_0\in I$ such that $f'(t_0)=0$ then, since $f''\leq
0$, $t_0$ is a global maximum of $f'$ and $\sup f(\tau)<\infty$.
Taking this into account, we get as a consequence of Theorem
\ref{maximales} the following result:

\begin{coro}
Let $\overline{M}= I \times_f F$ be a Lorentzian warped product with
dimension $n+1>2$ satisfying the TCC. If there exists a maximal
slice in $\overline{M}$, then every complete parabolic maximal
hypersurface whose hyperbolic angle is bounded and which is
contained in an open slab, must be a slice.
\end{coro}

Finally, we have the following non-existence result for
hypersurfaces in Lorentzian products.

\begin{teor}\label{otro}
Let $\overline{M}= I \times F$ be a Lorentzian product with
dimension $n+1>2$ whose fiber has non-negative Ricci curvature. Then
there is no complete parabolic CMC spacelike hypersurface in
$\overline{M}$ with bounded hyperbolic angle and constant mean
curvature $H\neq 0$.
\end{teor}
\begin{proof}
Suppose that there exists such a hypersurface $M$. Since $F$ has
non-negative Ricci, we get from (\ref{ARS1995}) that
\[
\Delta\cosh\varphi=\cosh\varphi\,{\rm
Ric}^F(N_F,N_F)+\cosh\varphi\,{\rm trace}\,(A^2)\geq 0.
\]
Thus, $\cosh\varphi$ is a positive and bounded subharmonic function
on a parabolic Riemannian manifold, and so $\cosh\varphi$ is
constant on $M$. But then ${\rm trace}\,(A^2)=0$ and $M$ is totally
geodesic, which is a contradiction. \hfill{$\Box$}
\end{proof}

The following Lemma is a consequence of the generalized maximum
principle for Riemannian manifolds given by Omori \cite{Om} (see
also Yau's paper \cite{Ya}):
\begin{quote}{\it
Let $M$ be a complete Riemannian manifold whose Ricci curvature is
bounded away from $-\infty$ and let $u:M\longrightarrow \mathbb{R}$
be a smooth function on $M$.
\begin{itemize}
\item[a)] If $u$ is bounded from above on $M$, then for each
$\varepsilon>0$ there exists a point $p_\varepsilon\in M$ such that
\[
|\nabla u(p_\varepsilon)|<\varepsilon, \quad \Delta
u(p_\varepsilon)<\varepsilon, \quad {\rm sup} \ u-\varepsilon <
u(p_\varepsilon)\leq {\rm sup} \ u;
\]
\item[b)] If $u$ is bounded from below on $M$, then for each
$\varepsilon>0$ there exists a point $p_\varepsilon\in M$ such that
\[
|\nabla u(p_\varepsilon)|<\varepsilon, \quad \Delta
u(p_\varepsilon)>-\varepsilon, \quad {\rm inf} \ u\leq
u(p_\varepsilon)< {\rm inf} \ u +\varepsilon.
\]
\end{itemize}
Here $\nabla u$ and $\Delta u$ denote, respectively, the gradient
and the Laplacian of $u$.}
\end{quote}

\begin{lema}\label{le1}
Let $\psi: M \rightarrow \overline{M}$ be an $n$-dimensional ($n\geq
2$) CMC complete spacelike hypersurface in a Lorentzian warped
product $\overline{M}= I \times_f F$ whose warping function
satisfies $(\log f)''\leq 0$. If the Ricci curvature of $M$ is
bounded from below and $M$ is contained between two slices, then
\[
H=-\frac{f'(\tau)}{f(\tau)}.
\]
\end{lema}
\begin{proof}
Since $M$ is contained between two slices, the function $\tau$ is
bounded from above and from below.

As $\tau$ is bounded from below, we know from the generalized
maximum principle that for each positive integer $m$ there exists a
point $p_m\in M$ such that
\[
-\frac{1}{m}>\Delta\tau(p_{m})=-\frac{f'(\tau(p_{m}))^2}{f(\tau(p_{m}))}\{n+\vert\nabla\tau
(p_{m})\vert^2\}-nH\overline{g}(N(p_{m}),\partial_t(p_{m}))
\]
where we have used (\ref{laptau}). Letting $m$ tend to infinity, and
taking into account that $\overline{g}(N(p_{m}),\partial_t(p_{m}))$
tends to 1 when $\vert\nabla\tau(p_{m}) \vert^2$ tends to 0, we get
\[
H\leq -\frac{f'(\inf \tau)}{f(\inf \tau)}.
\]

Analogously, since $\tau$ is bounded from above it follows that
\[
H\geq -\frac{f'(\sup \tau)}{f(\sup \tau)}.
\]

Finally, observe that the assumption $(\log f)''\leq 0$ means that
$\frac{-f'}{f}$ is an increasing function, and so
\[
H\leq -\frac{f'(\inf \tau)}{f(\inf \tau)}\leq -\frac{f'(\sup
\tau)}{f(\sup \tau)}\leq H
\]
which finishes the proof. \hfill{$\Box$}
\end{proof}

It worths pointing out that Lemma \ref{le1} goes a bit further than
\cite[Corollary 5.3]{C-Co-Ru-2} where, under similar assumptions for
the 2-dimensional case, a bound for $H^2$  is provided.

Note also that in Lemma \ref{le1} the hypothesis on $f$ follows
immediately if the ambient space satisfies the Timelike Convergence
Condition. On the other hand, the lower bound for the Ricci
curvature of the hypersurface can be deduced from suitable
assumptions on the sectional curvature of the fiber, as the
following Lemma shows:

\begin{lema}\label{le2}
Let $\psi: M \rightarrow \overline{M}$ be an $n$-dimensional ($n\geq
2$) CMC complete spacelike hypersurface in a Lorentzian warped
product $\overline{M}= I \times_f F$ whose fiber $F$ has
non-negative sectional curvature and its warping function satisfies $(\log f)''\leq 0$. Then the Ricci curvature of $M$ is
bounded from below.
\end{lema}
\begin{proof}
Given $p\in M$, let us take a local orthonormal frame
$\{U_1,...,U_n\}$ around $p$. From the Gauss equation
\[
\langle R(X,Y)V,W\rangle=\langle \overline{R}(X,Y)V,W\rangle+\langle
AY,W\rangle \langle AX,V \rangle - \langle AY,V\rangle \langle
AX,W\rangle, \quad X,Y,V,W\in \mathfrak{X}(M)
\]
where $\overline{R}$ and $R$ denote the curvature tensors of
$\overline{M}$ and $M$ respectively, and $A$ is the shape operator
of $\psi$, we get that the Ricci curvature of $M$, ${\rm Ric}^M$
satisfies
\[
{\rm Ric}^M(Y,Y)\geq \sum_k
\overline{g}(\overline{R}(Y,U_k)Y,U_k)-\frac{n^2}{4}H^2\vert
Y\vert^2, \quad Y\in \mathfrak{X}(M).
\]
Now, from \cite[Proposition 7.42]{O'N} we have
\begin{eqnarray*}
\sum_{k=1}^n \overline{g}(\overline{R}(Y,U_k)Y,U_k)&&=
\sum_{k=1}^ng(R^F(Y^F, U_k^F)Y^F,U_k^F)+ (n-1)\frac{f'^2}{f^2}\vert
Y\vert^2 \\
&& -(n-2)(\log f)''g(Y,\nabla\tau)^2 -(\log
f)''\vert\nabla\tau\vert^2\vert Y\vert^2\geq 0,
\end{eqnarray*}
where $R^F$ denotes the curvature tensor of the fiber $F$ and $Y^F$,
$U_k^F$ are the projections of $Y$, $U_k$ on the fiber $F$.
Therefore
\[
{\rm Ric}^M(Y,Y)\geq -\frac{n^2}{4}H^2\vert Y\vert^2,
\]
which ends the proof.\hfill{$\Box$}
\end{proof}

From Theorem \ref{t4} and Lemmas \ref{le1} and \ref{le2}, we get the
following result:

\begin{teor}\label{between}
Let $\overline{M}= I \times_f F$ be a Lorentzian warped product with
dimension $n+1\leq 5$, whose fiber has non-negative sectional
curvature. Let $M$ be a complete parabolic CMC spacelike
hypersurface  with mean curvature $H\not= 0$ which is contained
between two slices and whose hyperbolic angle is bounded. If the
warping function satisfies $f''(\tau)\leq 0$ on $M$, then $M$ must
be a spacelike slice.
\end{teor}
\begin{proof}
It is enough to observe that, under these assumptions, (\ref{RF})
holds. The proof finishes as in Theorem \ref{t4} using also Lemmas
\ref{le1} and \ref{le2}. \hfill{$\Box$}
\end{proof}

\begin{rem} Note that the boundedness of the sectional curvature of
a manifold does not assure its parabolicity. For instance, the
Euclidean space is not parabolic, although it has zero constant
sectional curvature. Neither, as shows \cite[Counterexample
5.4]{RRS}, the parabolicity of the manifold assures the boundedness
of the curvature.
\end{rem}

\begin{rem} Regarding the hypothesis in Theorem \ref{between} note that, in principle,
there does not exist any relation between the conditions "bounded
hyperbolic angle" and "being contained in a slab" for a spacelike
hypersurface in a Lorentzian warped product. In fact, every non
horizontal spacelike hyperplane in the Lorentz-Minkowski space has
bounded hyperbolic angle, but it is not contained in any slab.
Conversely, as shows \cite[Remark 5.3]{CRR2}, there exist spacelike
hypersurfaces in Lorentzian warped products which are contained in a
slab, but whose hyperbolic angle is not bounded.
\end{rem}

\section{Calabi-Bernstein type Problems} \label{CBr}

Let $(F,g_{_F})$ be a (non-compact) $n$-dimensional Riemannian
manifold and $f : I \longrightarrow \mathbb{R}$ a positive smooth
function. For each $u \in C^{\infty}(F)$ such that $u(F)\subset I$,
we can consider its graph $\Sigma_u=\{(u(p),p) \, : \, p\in F\}$ in
the Lorentzian warped product $(\overline{M}=I\times_f
F,\overline{g})$. The graph inherits from $\overline{M}$ a metric,
represented on $F$ by
$$g_u=-du^2+f(u)^2g_{_F}.$$
This metric is Riemannian (i.e., positive definite) if and only if
$u$ satisfies $|Du|<f(u)$, everywhere on $F$, where $Du$ denotes the
gradient of $u$ in $(F,g_{_F})$ and $| Du|^2=g_{_F}(Du,Du)$. Note
that $\tau(u(p),p)=u(p)$ for any  $p \in F$, and so $\tau$ and $u$
may be naturally identified on $\Sigma_u$. When $\Sigma_u$ is
spacelike, the unitary normal vector field on $\Sigma_u$ satisfying
$\overline{g}( N,\partial_t)>0$ is
\begin{equation}\label{Ng}
N=-\frac{1}{f(u)\sqrt{f(u)^2-\mid D u\mid^2}}\,\left(\,
f(u)^2\partial_t + Du \,\right),
\end{equation}
and the corresponding mean curvature function is
\[
H(u)=-\, \mathrm{div}\,\left(\frac{Du}{nf(u)\sqrt{f(u)^2-\mid
Du\mid^2}}\right) -\frac{f'(u)}{n\sqrt{f(u)^2 -\mid
Du\mid^2}}\left(n\,+\,\frac{\mid Du\mid^2}{f(u)^2}\right).
\]

In what follows, we will apply the previous uniqueness and non
existence results on CMC spacelike hypersurfaces given in Section
\ref{pr} to study the \emph{entire} solutions of the CMC spacelike
hypersurface equation in Lorentzian warped product $\overline{M}= I
\times_f F$ whose fiber is a parabolic Riemannian manifold.
Concretely, we will determine, in several relevant cases, all the
entire solutions of
\[
\mathrm{ div}\,\left(\frac{ D u}{f(u)\sqrt{f(u)^2-\mid
 D u\mid^2}}\right)=-nH-\frac{f'(u)}{\sqrt{f(u)^2-\mid
 D u\mid^2}}\left(n+\frac{\mid  Du\mid^2}{f(u)^2}\right),
\eqno\mathrm{(E.1)}
\]
\[
\hspace{2cm} \hspace*{3mm}\mid Du\mid<\lambda f(u), \quad
0<\lambda<1.\hspace*{73mm}\eqno\mathrm{(E.2)}
\]

Note that the constraint (E.2) can be written as
\begin{equation}\label{const}
\cosh \varphi < \,\frac{1}{\sqrt{1-\lambda^2}},
\end{equation}
where $\varphi$ is the hyperbolic angle of $\Sigma_u$. Therefore,
(E.2) implies that $\Sigma_u$ has bounded hyperbolic angle.
Moreover, this constraint means that the differential equation (E.1)
is, in fact, uniformly elliptic.

Now recall that the induced metric on a closed spacelike
hypersurface in a complete Lorentzian manifold could be non
complete. Actually, there exist entire spacelike graphs in
$\mathbb{L}^{n+1}$ which are non complete \cite{AM}. However, an
entire CMC spacelike graph in $\mathbb{L}^{n+1}$ must be complete
\cite{CY}. Therefore, if we want to derive a non-parametric
uniqueness result from a parametric one, we have to prove previously
the completeness of the induced metric. Indeed, the assumption on
$\lambda$ in (E.2) jointly with the fact that $f(u)$ does not
approach to zero on $F$, will provide the completeness of $g_u$ on
$F$ if $g_{_F}$ is complete, as we prove in the following result.

\begin{lema}\label{complete} Let $\overline{M}=I\times_f F$ be a Lorentzian warped product whose fiber
is a (non-compact) complete Riemannian manifold. Consider a function
$u\in C^{\infty}(F)$, with ${\rm Im}(u)\subseteq I$, such that the
entire graph $\Sigma_u=\{(u(p),p) \, : \, p\in F\}\subset
\overline{M}$ endowed with the metric $g_u=-du^2+f(u)^2g_{_F}$ is
spacelike. If the hyperbolic angle of $\Sigma_u$ is bounded and
$\inf f(u)>0$, then the graph $(\Sigma _u,g_{_{\Sigma_u}}$ is
complete, or equivalently the Riemannian surface $(F,g_u)$ is
complete.
\end{lema}
\begin{proof}
The classical Schwartz inequality gives
\[
g_{_{\Sigma_u}}(\nabla\tau, v)^2\leq
g_{_{\Sigma_u}}(\nabla\tau,\nabla\tau) \, g_{_{\Sigma_u}}(v,v), \ \
{\rm for \ \ all}\ \ v\in T_q(\Sigma_u), \,\, q\in \Sigma_u
\]
and therefore
\[
g_{_{\Sigma_u}}(v,v)\geq -g_{_{\Sigma_u}}(\nabla\tau,\nabla\tau) \,
g_{_{\Sigma_u}}(v,v)+ f(\tau)^2 g_{_F}(d\pi_{_F}(v),d\pi_{_F}(v)),
\]
which implies
\[
g_{_{\Sigma_u}}(v,v)\geq\frac{f(\tau)^2}{\cosh^2\varphi} \,
g_{_F}(d\pi_{_F}(v),d\pi_{_F}(v)),
\]
and $\sup(\cosh\varphi)<\infty$. If we denote by ${\cal L}(\alpha)$
and $\mathcal{L}_u(\alpha)$ the lengths of a smooth curve $\alpha$
on $F$ with respect to the metrics $g_{_F}$ and $g_u$ ,
respectively, it is not difficult to see that
$${\cal L}_u(\alpha)\geq B\inf (f(u)){\cal L}_u(\alpha),$$
\noindent where $B=\frac{1}{\sup(\cosh\varphi)}$. Therefore, since
the Riemannian manifold $(F,g_{_F})$ is complete and $\inf
(f(u))>0$, we conclude that the metric $g_u$ is also complete.
\hfill{$\Box$}
\end{proof}

Next, making use of the study developed in Section \ref{pr}, we
provide several results for bounded solutions to the Equation
(E.1)+(E.2) under suitable assumptions.

\begin{teor} Let $(F,g)$ be a simply connected complete parabolic Riemannian $n$-manifold,
$n\leq 4$, whose sectional curvature is non-negative. Let
$f:I\longrightarrow \mathbb{R}^+$ be a smooth function satisfying
$f''\leq 0$. Then, the only bounded entire solutions to the Equation
(E.1)+(E.2) for $H\in \mathbb{R}^*$, are the constant functions
$u=u_0$ with $H=-\frac{f'(u_0)}{f(u_0)}$.
\end{teor}
\begin{proof}
Let $u$ be an entire solution to (E.1)+(E.2). As we have commented
above, the normal unitary vector field on the graph
$\Sigma_u=\{(u(p),p) \, : \, p\in F\}$ satisfying that
$\cosh\varphi=\overline{g}(N,\partial_t)>0$ is given by (\ref{Ng}),
and hence the constraint (E.2) can be expressed as (\ref{const}).
Finally, making use of Lemma \ref{complete} and Theorem
\ref{between} the proof ends. \hfill{$\Box$}
\end{proof}

As a consequence of Lemma \ref{complete} and Theorem \ref{none}, we
have
\begin{teor}
Let $(F,g)$ be a simply connected complete parabolic Riemannian
$n$-manifold, $n> 4$. Let $f:I\longrightarrow \mathbb{R}^+$ be a
smooth function satisfying $f''\leq 0$, $\inf f>0$, $\sup f<\infty$
and $$\inf{\rm Ric}^F\geq n(ff''-f'^2).$$ If $H\in\mathbb{R}^*$ is
such that
\[
H^2\leq \inf_{I}\left\lbrace \frac{f'(t)^2}{f(t)^2}\right\rbrace,
\]
then the Equation (E.1)+(E.2) has not entire solutions.
\end{teor}

\begin{coro}
Let $(F,g)$ be a simply connected complete parabolic Riemannian
$n$-manifold, $n> 4$. Let $f:I\longrightarrow \mathbb{R}^+$ be a
smooth function satisfying $f''\leq 0$ and $$\inf{\rm Ric}^F\geq
n(ff''-f'^2).$$ If $H\in\mathbb{R}^*$ is such that
\[
H^2\leq \inf_{I}\left\lbrace \frac{f'(t)^2}{f(t)^2}\right\rbrace,
\]
then the Equation (E.1)+(E.2) has not entire bounded solutions.
\end{coro}

When $f=1$, we have
\begin{teor}
Let $(F,g)$  be a simply connected complete parabolic Riemannian
$n$-manifold, $n\geq 2$, whose Ricci curvature is non negative. Then
there is no entire solutions to the equation
\[
\mathrm{ div}\,\left(\frac{ D u}{\sqrt{1-\mid
 D u\mid^2}}\right)=-nH
\]
\[
\mid Du\mid<\lambda, \quad 0<\lambda<1,
\]
where $H\not=0$.
\end{teor}
\begin{proof} It follows from Lemma \ref{complete} and Theorem
\ref{otro}. \hfill{$\Box$}
\end{proof}

We finish with some results under the assumption $H=0$, the second
one for $f=1$ constant.
\begin{teor}
Let $(F,g)$ a simply connected complete parabolic Riemannian
$n$-manifold, $n\geq 2$. Let $f:I\longrightarrow \mathbb{R}^+$ be a
smooth function satisfying $f''\leq 0$, $\inf f>0$ and $\sup
f<\infty$. Then the only entire solutions to the equation
\[
\mathrm{ div}\,\left(\frac{ D u}{f(u)\sqrt{f(u)^2-\mid
 D u\mid^2}}\right)=-\frac{f'(u)}{\sqrt{f(u)^2-\mid
 D u\mid^2}}\left(n+\frac{\mid  Du\mid^2}{f(u)^2}\right)
\]
\[
\mid Du\mid<\lambda f(u), \quad 0<\lambda<1,
\]
which are bounded from above or from below are the
constants $u=u_0$, with $f'(u_0)=0$.
\end{teor}
\begin{proof} It follows from Lemma \ref{complete} and Theorem
\ref{maximales}. \hfill{$\Box$}
\end{proof}

\begin{coro}
Let $(F,g)$ a simply connected complete parabolic Riemannian
$n$-manifold, $n\geq 2$. Let $f:I\longrightarrow \mathbb{R}^+$ be a
smooth function satisfying $f''\leq 0$. Then the only bounded entire
solutions to the equation
\[
\mathrm{ div}\,\left(\frac{ D u}{f(u)\sqrt{f(u)^2-\mid
 D u\mid^2}}\right)=-\frac{f'(u)}{\sqrt{f(u)^2-\mid
 D u\mid^2}}\left(n+\frac{\mid  Du\mid^2}{f(u)^2}\right)
\]
\[
\mid Du\mid<\lambda f(u), \quad 0<\lambda<1,
\]
are the constants $u=u_0$, with $f'(u_0)=0$.
\end{coro}

If, moreover, $f=1$, we have
\begin{coro}
Let $(F,g)$ be a simply connected complete parabolic Riemannian
$n$-manifold, $n\geq 2$, whose Ricci curvature is non negative. Then
the only entire solutions to the equation
\[
\mathrm{ div}\,\left(\frac{ D u}{\sqrt{1-\mid
 D u\mid^2}}\right)=0
\]
\[
\mid Du\mid<\lambda, \quad 0<\lambda<1,
\]
which are bounded from above or below are the constants.
\end{coro}
\begin{proof} It follows from Lemma \ref{complete} and Corollary
\ref{maximales2}. \hfill{$\Box$}
\end{proof}

\begin{ejem} Consider the Riemannian
product $\mathbb{S}^ 2\times \mathbb{R}^ 2$ of the sphere with its
usual metric and the Euclidean plane. Then the only entire solutions
which are bounded from above or below to the equation
\[
\mathrm{
div}\,\Big(\frac{ D u}{\sqrt{1-\mid
 D u\mid^2}}\Big)=-4H
\]
\[
\hspace{0.5cm} \mid Du\mid<\lambda, \quad 0<\lambda<1,
\] for
$H=0$, are the constants.

Moreover, if $H\in \mathbb{R}^*$, then there is no entire solutions.
\end{ejem}

\section*{Acknowledgments}
The first author is partially supported by the Spanish MICINN Grant
with FEDER funds MTM2010-19821. The second and third authors are
partially supported by the Spanish MICINN Grant with FEDER funds
MTM2010-18099. The third author is also partially supported by the
Junta de Andaluc\'\i a Regional Grant with FEDER funds P09-FQM-4496.

\end{document}